\documentclass{amsart}

\usepackage{geometry,graphicx,amssymb,amsmath,amsbsy,eucal,amsfonts,mathrsfs,amscd,bm,color}
\usepackage[all]{xy}

\numberwithin{equation}{section}

\allowdisplaybreaks[4]

\newtheorem{theorem}{Theorem}[section]
\newtheorem{lemma}[theorem]{Lemma}
\newtheorem{corollary}[theorem]{Corollary}
\newtheorem{proposition}[theorem]{Proposition}
\theoremstyle{definition}
\newtheorem{definition}[theorem]{Definition}
\theoremstyle{remark}

\theoremstyle{definition}
\newtheorem{example}[theorem]{Example}

\newcommand{\bc}{{ c}} 
\newcommand{\bZ}{{\mathbb{Z}}}

\subjclass[2010]{Primary }
\title[]{Chip-firing and energy minimization on M-matrices}
\author[J.~Guzm\'an and C. Klivans]{{Johnny Guzm\'an}
 \email{johnny\_guzman@brown.edu}
 \address{Division of Applied Mathematics, Brown University, Providence, RI 02906}
\and{Caroline Klivans}
\email{klivans@brown.edu}
\address{Division of Applied Mathematics and Department of Computer Science, Brown University, Providence, RI 02906}
}
\date{\small \today}

 \thanks{}
 
 
 
\begin{document}
\maketitle

\begin{abstract}

We consider chip-firing dynamics defined by arbitrary M-matrices.
M-matrices generalize graph Laplacians and were shown by Gabrielov to
yield avalanche finite systems. Building on the work of Baker and
Shokrieh, we extend the concept of energy minimizing chip
configurations.  Given an M-matrix, we show that there exists a unique
energy minimizing configuration in each equivalence class defined by
the matrix.

We consider the class of $z$-superstable configurations.
% which satisfy
%a strictly stronger stability requirement than superstable
%configurations (equivalently $G$-parking functions or reduced divisors).
We prove that for any M-matrix, the $z$-superstable configurations
coincide with the energy minimizing configurations.  Moreover, we
prove that the $z$-superstable configurations are in simple duality
with critical configurations.  Thus for all avalanche-finite systems
(including all directed graphs with a global sink) there exist unique
critical, energy minimizing and $z$-superstable configurations.   The critical configurations are
in simple duality with energy minimizers which coincide with $z$-superstable configurations.

%% We consider chip-firing dynamics defined by arbitrary M-matrices.
%% M-matrices generalize graph Laplacians and were shown by Gabrielov to
%% yield avalanche finite systems. Building on the work of Baker and
%% Shokrieh, we extend the concept of energy minimizing chip
%% configurations.  Given an M-matrix, we show that there exists a unique
%% energy minimizing configuration in each equivalence class defined by
%% the matrix.

%% We define the class of $z$-superstable configurations which satisfy
%% a strictly stronger stability requirement than superstable
%% configurations (equivalently $G$-parking functions or reduced divisors).  We
%% prove that for any M-matrix, the $z$-superstable configurations
%% coincide with the energy minimizing configurations.  Moreover, we
%% prove that the $z$-superstable configurations are in simple duality
%% with critical configurations.  Thus for all avalanche-finite systems
%% (including all directed graphs with a global sink) there exist unique
%% critical, energy minimizing and $z$-superstable configurations.   The critical configurations are
%% in simple duality with energy minimizers which coincide with $z$-superstable configurations.

%\vspace{.5in}
%\noindent Corresponding author: Caroline Klivans\\
%Telephone: 401-863-1460 / Fax: 401-863-1355\\
%Email: klivans@brown.edu\\

\end{abstract}

%%%%%%%%%%%%%%%%%%%%%%%%%%%%%%%%%
%%%%%%%%%%%%%%%%%%%%%%%%%%%%%%%%%
%%%%%%%%%%%%%%%%%%%%%%%%%%%%%%%%%
\section{Introduction}

There is a large literature on the dynamics and combinatorics of
chip-firing games.  They were originally studied in the context of
self-organized criticality and sandpile models \cite{BTW, Dhar, gabe}, as balancing games on graphs \cite{BLS, Spencer}, and for their algebraic structure \cite{Biggs, CRS}.
More recently, chip-firing has appeared in a surprising variety of new
connections.  For example, chip firing plays a central role in a
Riemann-Roch theorem for graphs \cite{BN} and linear systems in
tropical geometry \cite{HMY}.  Our starting point will be the recent
work of Baker and Shokrieh on chip-firing, potential theory and energy
minimization on graphs \cite{BS}.  Building on this new connection to
energy minimization, we are able to return to some of the first questions
concerning the long-term stability of chip-firing dynamics.

There are many variants to the chip-firing game.  Typically one considers a
finite graph with integer values associated to the vertices.  A single
vertex is distinguished as the sink (or bank).  The value of the sink
may be arbitrary but all other vertices have non-negative values,
which we think of as the number of chips associated to the vertex.  A
chip firing rule is given as follows: if any non-sink vertex has at
least as many chips as it has neighbors, then it ``fires'' by sending
one chip to each of its neighbors. The value of each neighbor is
increased by one and the value at the vertex that fired is decreased
by its degree.  In particular, if we consider the number of chips at
each vertex as an integer vector, called a chip configuration, then
``firing'' a vertex subtracts the corresponding row of the graph
Laplacian from the configuration.  Two chip configurations are
considered equivalent if their difference is in the image of the graph
Laplacian. Informally, two chip configurations are equivalent if  one chip configuration can transform to the
other via fires and reverse-fires.

Of great interest is the long term behavior of such systems.
If the system has a sink, as above, then every configuration does
stabilize in the sense that eventually no non-sink vertex will be able
to fire.  Imposing further stability requirements leads to important
classes of chip configurations.  Briefly, {\emph{superstable}}
configurations (also known as $G$-parking functions or reduced divisors)
are stable configurations such that no subset of vertices can
simultaneously fire and result in a non-negative
configuration. \emph{Critical} configurations (also known as recurrent
configurations) are stable configurations that can be reached from
sufficiently large starting configurations.  It is well known that
superstable configurations and critical configurations exist and are
unique per equivalence class of chip configurations.  Furthermore,
critical and superstable configurations are in simple duality with each other.  They are also in bijection with the
spanning trees of the graph; see e.g. \cite{CP} or \cite{Tetali}.

Baker and Shokrieh \cite{BS} introduced a norm on chip configurations in terms of the graph Laplacian for undirected graphs.  The norm
is thought of as an energy function and they
 investigated energy-minimizing chip configurations.  In particular, they prove
that energy-minimizers are precisely the superstable configurations
and hence unique per equivalence class and in duality with critical
configurations.

Following the work of Dhar \cite{Dhar}, Gabrielov \cite{gabe} considered more general
chip-firing dynamics in terms of a class of dissipation matrices which is broader 
than the graph Laplacians above.  He worked with 
\emph{avalanche-finite matrices}, which precisely guarantee that all
configurations eventually stabilize using legal firing moves (see Section \ref{sec:M} for specifics).  Gabrielov
proved that critical configurations exist and are unique per
equivalence class for all avalanche-finite matrices.

More recently, much attention has focused on the intermediate case of
chip-firing on directed graphs.   In this case, the existence and
uniqueness of critical configurations is guaranteed by Gabrielov's
earlier work, because the associated graph Laplacians are a special
case of avalanche-finite matrices. In this setting, the term superstable is used in at least two
different ways.  We will use the notation $\chi$- and $z$- superstable
to distinguish the classes of configurations (see Section
\ref{stable}).  The uniqueness of $\chi$-superstable configurations and
the duality with critical configurations appears in \cite{lionel, PS}
for the special case of Eulerian directed graphs.  For all directed
graphs with a global sink, a stronger form of stability is required
for an analogous result.  The uniqueness of $z$-superstable
configurations and the duality with critical configurations appears
originally in \cite{Perlman} and later in \cite{Perkinson} and
\cite{Asadi}.

% had only been observed to hold for the
%special case of Eulerian directed graphs \cite{lionel, PS}.

%% But, the uniqueness of the
%% established notion of superstable configuration had only been observed to hold for the
%% special case of Eulerian directed graphs \cite{lionel, PS}.

We unify and generalize these results as follows.  First, building on Baker and
Shokrieh's work, we define a class of norms and energy-minimizing
configurations for all avalanche-finite matrices.  We prove the
existence and uniqueness of energy-minimizers per equivalence class of
these matrices (Theorem~\ref{uniquem}). 
%%  Second, we generalize the definition of a superstable
%% configuration to what we call $z$-superstable configurations.  In the
%% undirected graph case and the Eulerian directed graph case, they
%% coincide with the classic notion.  For a general avalanche-finite
%% matrix, however, they impose a strictly stronger stability condition.
We show that the $z$-superstable configurations are precisely the
energy-minimizers (Theorem~\ref{thm1}) and are in simple duality with critical
configurations (Theorem~\ref{thm:duality}).

%%  Second, we generalize the definition of a superstable
%% configuration to what we call $z$-superstable configurations.  In the
%% undirected graph case and the Eulerian directed graph case, they
%% coincide with the classic notion.  For a general avalanche-finite
%% matrix, however, they impose a strictly stronger stability condition.
%% We show that these $z$-superstable configurations are precisely the
%% energy-minimizers (Theorem~\ref{thm1}) and are in simple duality with critical
%% configurations (Theorem~\ref{thm:duality}).

 Namely, for all avalanche-finite matrices, there exist unique
 critical, energy-minimizing, and $z$-superstable configurations. The
 first are in simple duality with the latter two which coincide.  The
 number of such configurations is given by the determinant of the
 matrix.

%%%%%%%%%%%%%%%%%%%%%%%%%%%%%%%%
\section{M-matrices}\label{sec:M}

The dynamics of chip-firing on graphs is dictated by the reduced graph
Laplacian.  Let $G$ be a directed (multi)-graph with $n+1$ vertices. The graph Laplacian $\tilde{\Delta}(G)$ is given by
 \[ \tilde{\Delta}_{ij}=\begin{cases} -a_{ij} & \text{$i\neq j$ and $(i,j)\in E$} \\ \text{outdeg}(i) & \text{$i=j$} \\ 0 & \text{otherswise}, \end{cases}\] 
%% $$
%% \tilde{\Delta}_{ij} = \left\{
%%         \begin{array}{ll}
%%             -a_{ij} & \quad i \neq j \text{ and } (i,j) \in E  \\
%%             \text{outdeg}(i) & \quad  i=j \\
%%               0 & \quad \text{otherwise},
%%         \end{array}
%%     \right.
%% $$
    where $a_{ij}$ is the number of edges from $i$ to $j$.
A reduced graph Laplacian is any matrix resulting from deleting a single row and column from a graph Laplacian.  
All graphs we will consider will have a global sink.  A graph $G$ has a global sink, $s$, if for every vertex $v \neq s$ there is a directed path from $v$ to  $s$.  When referring to the reduced Laplacian $\Delta(G)$ for a graph with a global sink, we will always assume the row and column corresponding to the sink has been deleted.

  Given a graph on $n+1$ vertices with the last vertex a global sink, a chip
configuration ${ c} = (c_1, c_2, \ldots, c_n)$ is a non-negative
integer vector, $\bc \in \mathbb{Z}^n_{\geq 0}$.  The value $c_i$ is
thought of as the number of chips at vertex $i$.  Starting with a
configuration $\bc$, firing vertex $i$ results in subtracting the
$i$th row of the reduced Laplacian from ${ c}$; ${ c} - Le_i$ where $L=\Delta^T$ and $e_i$ is the $i$th standard basis vector in $\mathbb{R}^n$.  In this notation, the $n+1$st vertex is the sink vertex and we will not be concerned with its ``chip value''.  

Gabrielov considered more general chip-firing systems 
 by replacing the reduced graph Laplacian with a broader class of matrices \cite{gabe}.  
For an arbitrary $n \times n$ integer matrix N, we consider a system
with $n$ states.  A chip configuration is any integer vector ${ c}
\in \mathbb{Z}^n$.  Firing a state $i$ is defined to be the process
which replaces the configuration ${  c}$ with ${  c} - \textrm{N}^Te_i$,
namely subtracting the $i$th row of N.  Two configurations ${  c}$
and ${  d}$ are considered equivalent if 
their difference ${  c} - {  d}$ is in the
$\mathbb{Z}$-image of N.  In this more general setup, a state $i$ is
allowed to fire if $c_i \geq \textrm{N}_{ii}$.  Following the physicality of
the original model, Gabrielov restricted to matrices with a positive
diagonal and non-positive off-diagonal.  Therefore, a state must have
a certain positive amount of chips in order to fire and firing a state
increases the number of chips on neighboring states.  These are
referred to as redistribution matrices in ~\cite{gabe}.  A configuration
is stable if $c_i < \textrm{N}_{ii}$  for all states $i$. A natural question arises:
for which such matrices does the chip-firing process eventually stabilize versus
producing an infinite process.  
An \emph{avalanche-finite} matrix is one for which every non-negative chip
configuration  stabilizes.  We will look closely at such
matrices.
We start with some definitions.

\begin{definition}
An $n\times n$ matrix $L$ such that  $L_{ij} \le 0$ for all $i \ne j$ is called a Z-matrix.
%Let $L$ be a $n \times n$ matrix then if $L_{ij} \le 0$ for all $i \ne j$,  $L$ is called a Z-matrix.   
\end{definition}

\begin{definition}
Let $L$ be a $n \times n$ Z-matrix. If any of the following equivalent conditions hold then $L$ is called a non-singular M-matrix:

\begin{enumerate}
\item $L$ is avalanche finite.
\item The real part of the eigenvalues are positive. 
\item $L^{-1}$ exists and all the entries of $L^{-1}$ are non-negative.
\item There exists a vector $x \in \mathbb{R}^n$ with $ x \geq 0$ such that $L x$ has all positive entries.
\end{enumerate}
\end{definition}

The equivalence of the last three conditions can be found, for example, in Plemmons \cite{Plemmons}.  The equivalence of first condition is due to Gabrielov \cite{gabe}.  M-matrices appear in many different fields including economics, operations research, finite difference and finite element analysis ; see for example \cite{bramblehubbard, BurmanErn, CiarletRaviart, kaneko, Leontief, XuZikatanov}.  In particular, if the stiffness matrix (e.g. the discrete Laplacian) of the finite element method is an M-matrix then the solution satisfies a discrete maximum principle \cite{BurmanErn, CiarletRaviart,  XuZikatanov}. The discrete maximum principle was an important property used by Baker and Shokrieh \cite{BS} in their work on the graph Laplacian for undirected graphs.

We note that these conditions do not necessitate that
M has either positive row or column sums.  The desired properties of
chip-firing such as the existence of unique critical, superstable and
energy minimizing configurations will all hold in this more general
setting.

\section{Energy Minimization}

In this section, building on work of Baker and Shokrieh \cite{BS}, we will
prove that for any M-matrix, energy-minimizing configurations exist and are unique
per equivalence class.

Given an M-matrix $L$ and an integer vector $q$ define the following energy,
\begin{equation}\label{2energy}
E(q)=\|L^{-1} q\|_{2}^2,
\end{equation}
where $\|v\|_2^2= v \cdot v$.  

This energy is different from the
form used in \cite{BS}.
They defined energy minimizers in terms of a
norm using any pseudo-inverse of the (undirected) graph Laplacian. 
One of the reasons to choose the energy form here is that it
allows us to consider non-symmetric matrices $L$.  This will be particularly 
important in the directed graph case. 

The energy form can be extended to any energy from the class
$E(q)=\|L^{-1} q\|_G^2$ where $\|v\|_G^2= v^t G v$ and G is a
symmetric positive definite matrix with $L^{-1} G \ge 0$.  All of our
results below hold for this more general setting.  Moreover, if $L$ is
symmetric then setting $G = L$ recovers the energy used in \cite{BS}.
For ease of exposition, we use the simplest form where $G = I$. 

%\revj{
Furthermore, there are more energies one can consider. For instance, instead of basing the energy on the 2-norm, as in \eqref{2energy}, one can base the energy on a $p$-norm:
\begin{equation}\label{penergy}
E(q)=\|L^{-1} q\|_{p}^p,
\end{equation}
where $\|D\|_{p}^p=\sum_{i=1}^n |D_i|^p.$ The case $p=1$ was considered by Baker and Shokrieh \cite{BS} where they called this quantity the potential. In fact, all of our results will hold for energies defined in the following way
\begin{equation}\label{generalenergy}
E(q)=\sum_{i=1}^n \phi_i((L^{-1} q)_i),
\end{equation}
where the functions $\phi_i: \mathbb{R} \rightarrow \mathbb{R}$ are non-negative and strictly increasing. For example, $\phi_i(x)=|x|^p$ for all $1 \le i \le n$ as in the case of \eqref{penergy}. Another example, could be $\phi_i(x)=\log(1+|x|)$ for all $i$.  In the appendix we prove that the main results of the paper hold for these more general energies. However, for simplicity,  up until the appendix we will restrict the discussion to the energy \eqref{2energy}.
% } 

Baker and Shokrieh work in the undirected graph case and show that
energy-minimizing configurations  are precisely
the superstable configurations (which they refer to as reduced divisors) of the graph.  Hence energy-minimizers exist and are
unique per equivalence class.  In the current section we work directly with the energy
minimization problem for arbitrary M-matrices.  We make the
connection to superstable configurations in Section~\ref{sec:chip}.

The energy minimization problem is posed on equivalence classes
induced by $L$.  In order to have cleaner notation, we will adopt the
following convention throughout: Let $\Delta$ be an M-matrix and $L
= \Delta^{T}$.  In this way, we will not have to continually write the
transpose for row operations.  Also note that the transpose of an M-matrix is always an M-matrix.  

\begin{definition}
Two configurations $f,g \in \mathbb{Z}^n$ are equivalent, denoted $f\sim g$, if $g-f=L z$ for some $z \in \mathbb{Z}^n$. 
The equivalence class of $f$ is denoted by $[f]$.
%=\{ g \in \mathbb{Z}^n: g \sim f\}$.  
\end{definition}

Given  $f \in \mathbb{Z}^n$ with  $f \ge 0$ consider the following problem:
\begin{equation}\label{min}
\min_{g\sim f, g \ge 0}  E(g).
\end{equation}
A solution to the minimization problem is a non-negative configuration
equivalent to $f$ with smallest energy.  We call such a configuration
an energy-minimizer.  Since we are working in a discrete space it is
not difficult to see that minimizers always exist.  We will prove that for any M-matrix
there is a \emph{unique} energy-minimizer per equivalence class. To do this,
we first prove two preliminary lemmas.  We need the following
notation. Given $z \in \mathbb{Z}^n$ define $z^+ \in \mathbb{Z}^n_{\ge
  0}$ by

$$
z^+_i = \left\{
        \begin{array}{ll}
            z_i & \quad \text{if } z_i \ge 0  \\
              0 & \quad \text{otherwise.}
        \end{array}
    \right.
$$
 Similarly, define $z^- \in \mathbb{Z}^n_{\le 0}$ by replacing all positive entries of $z$ with $0$.

\begin{lemma}\label{lemma1}
Let $L$ be a $Z$-matrix. If $f,g \ge 0$ and $g=f-Lz$ then $h=f-L z^+ \ge 0$.  
%Let L satisfy $L_{i j}  \le 0$ for any $i \ne j$. If $f,g \ge 0$ and $g=f-Lz$ then $h=f-L z^+ \ge 0$.  

\end{lemma}
\begin{proof}
Suppose that $z^+_i=0$. Then it is clear that $-(Lz^+)_i \ge 0$. Hence, $h_i \ge f_i \ge 0$. On the other hand suppose that $z^+_i >0$, then $z^{-}=z-z^{+}$ satisfies $z^{-}_i=0$ and so $(Lz^{-})_i \ge 0$, or equivalently $(Lz)_i \ge (Lz^{+})_i$ and so $f_i-(Lz^{+})_i \ge f_i-(Lz)_i \ge 0$.   
\end{proof}

\noindent The next Lemma expresses the difference in energy of two equivalent configurations.

\begin{lemma}\label{lemma2}
Let $L$ be an M-matrix and suppose that $g=f-Lz$, then
\begin{equation*}
E(g)=E(f)+z^t z-2z^t L^{-1}f=E(f)-z^tz-2z^tL^{-1}g. 
\end{equation*}
\end{lemma}
\begin{proof}
\begin{alignat*}{1}
E(g)&=\|L^{-1}(f-L z)\|_2^2\\
&=\|L^{-1} f-z\|_2^2 \\
&=\|L^{-1} f\|_{2}^2+z^t z -2 z^t L^{-1} f\\
&=E(f)+z^t z -2 z^t L^{-1} f\\
&=E(f)-z^t z-2z^t L^{-1} g.
\qedhere
\end{alignat*}
\end{proof}

\noindent We now state our main theorem for this section.

\begin{theorem}\label{uniquem}
Let $L$ be  an M-matrix. For every configuration $f$, there exists a unique energy minimizer equivalent to $f$. Namely, for every configuration $f$, there exists a unique solution to problem \eqref{min}. 
\end{theorem}

\begin{proof}
Suppose that $g \sim f$ and $w\sim f$ with $ g,w \ge 0$ both minimizers to problem \eqref{min}. We will show that $g=w$.
Because $g$ is equivalent to $w$,  there exists $z$ such that $g=w-Lz$ for some $z \in \bZ^n$. By Lemma \ref{lemma1} we know that $h=w-L z^+ \ge 0$ and of course $h \sim w \sim f$.  By Lemma \ref{lemma2} we have
\begin{alignat*}{1}
E(h)=&E(w)-(z^+)^t z^+-2(z^+)^tL^{-1}h.
\end{alignat*}

Using that $L^{-1}$ is a non-negative matrix and $h \ge 0$, $L^{-1}h\ge0$. This implies that $-2(z^+)^tL^{-1}h \le 0$, and so 
\begin{equation*}
E(h) \le E(w)-(z^+)^t z^+.
\end{equation*}
Since $w$ is a minimizer it must be that $z^+ =0$ or that $z \le 0$. 

On the other hand, we similarly have 
\begin{equation*}
E(w)=E(g)+z^tz-2 z^tL^{-1}w.
\end{equation*}
Since $z \le 0$ this shows that $E(g)< E(w)$ unless $z = 0$. 
\end{proof}

\section{Chip-firing on M-matrices}\label{sec:chip}

In Section~\ref{sec:M} we defined chip-firing on M-matrices.  For an $n \times
n$ M-matrix $\Delta$, we consider a system with $n$ states.  A configuration is
any integer vector ${  c} \in \bZ^n$, with $c_i$ considered the
number of chips at state $i$.   For a configuration ${ 
  c}$, state $i$ is allowed to fire if $c_i \geq \Delta_{ii}$ (recall that
M-matrices have non-negative diagonal entries).  The resulting
configuration is ${  c'} = {  c} - L e_i$ where $L=\Delta^T$.

In Section~\ref{stable} we exam three important types of chip
configurations - stable, $\chi$-superstable, and $z$-superstable.  In
Section~\ref{z-energy}, we prove that energy-minimizers coincide with
$z$-superstable configurations.  In Section~\ref{z-critical}, we prove
that $z$-superstable configurations are in duality with critical
configurations.

\subsection{Stability}\label{stable}

We consider three notions of `stable' configurations, each strictly stronger than the previous.
The definitions could be made with respect to any matrix, again we have in mind that $L$ is the transpose of an M-matrix.

\begin{definition}
A vector $f \in \bZ^n$  is stable if for all $i$, $f_i < L_{ii}$.
\end{definition}

\begin{definition}
A vector $f \in \bZ^n$ with $f \ge 0$ is $\chi$-superstable if for every $\chi \in \{0,1\}^n$ with  $\chi \neq 0$ there exists $1 \le i \le n$ such that  
\begin{equation*}
f_i-(L\chi)_i <0. 
\end{equation*}
\end{definition} 

\begin{definition}
A vector $f \in \bZ^n$ with $f \ge 0$ is  $z$-superstable if for every $z \in \bZ^n$ with $z \ge 0$ and $z \neq 0$  there exists  $1 \le i \le n$ such that  
\begin{equation*}
f_i-(Lz)_i <0. 
\end{equation*}
\end{definition}

The above notion of stable configuration is standard in the literature.
  A stable configuration is one in which no individual
state can fire. 
 A $\chi$-superstable configuration is one in which no
subset of states can simultaneously fire and result in a non-negative
configuration.   A $z$-superstable configuration is one in which no multiset of states can simultaneously fire and result in a non-negative configuration.  

In \cite{lionel}, the term superstable is used for
$\chi$-superstables.  In \cite{Perkinson}, the term superstable is
used for $z$-superstables.  For undirected graphs and Eulerian
directed graphs, the notions coincide, i.e. a configuration is
$\chi$-superstable if and only if it is $z$-superstable.  Moving to
non-Eulerian directed graphs and more generally to M-matrices,  the
$z$-superstable condition is strictly stronger.
It is immediately clear that if $f$ is $z$-superstable then it is $\chi$-superstable.  The following result gives sufficient conditions on a matrix for the converse to hold.
\begin{theorem}\label{thm:same}
Suppose $L$ is a matrix with non-positive off diagonal entries and {\bf non-negative row sums}\footnote{In the graphical case, $L$ is the transpose of the reduced graph Laplacian. Hence this result applies to graphs whose Laplacians have non-negative column sums.}. Then, if $f \ge 0$ is $\chi$-superstable it is  $z$-superstable. 
\end{theorem}
\begin{proof}
Let $z \in \bZ^n$ with $z \ge 0$ and $z \neq 0$. We will show that $f-Lz$ must have a negative entry. To this end, let $\kappa=\max_{i} z_i$ and so $\kappa >0$. Define $\chi \in \{0,1\}^n$ such that $\chi_i=1$ if $z_i >0$ and $\chi_i=0$ if $z_i \le 0$.  Let $\tilde{z}=z-\chi$ and let $\tilde{\kappa}= \max_{i} \tilde{z}_i$.   Then  $\tilde{z}_j=\tilde{\kappa}$ for every $j$ such that $\chi_j=1$. Using that $L$ has non-positive off diagonal entries we have that $(L \chi)_i <0$ for every $i$ such that $\chi_i=0$. Therefore, $f_i-(L \chi)_i \ge f_i \ge 0$ for $i$  such that $\chi=0$. Since $f$ is $\chi$-superstable this means that $f_j -(L \chi)_j<0$ for some $j$ where $\chi_j=1$. Consider such a $j$ then we argued that  $\tilde{z}_j=\tilde{\kappa}$ and so
\begin{equation*}
 (L \tilde{z})_j =L_{j j} \tilde{z_j}+\sum_{i \neq j} L_{ji} \tilde{z}_i \ge  \tilde{\kappa} \sum_{i} L_{ji} \ge 0,
\end{equation*}
since we are assuming the row sums are non-negative.

Hence,
\begin{equation*}
f_j-(Lz)_j=f_j-(L \chi)_j-(L \tilde{z})_j< f_j-(L \chi)_j<0.
\qedhere
\end{equation*}

\end{proof}

The next example shows that if the non-negative row sum condition is dropped, then $\chi$-superstable configurations may not be $z$-superstable and not unique per equivalence class.

\begin{example}
Consider the following M-matrix which does not have positive row or column sums:
\[
L = \begin{pmatrix}
3 & -4 \\
-1 & 2 
\end{pmatrix}.
\]
An explicit calculation shows the image of the three non-zero characteristic vectors: 
\[ L \begin{pmatrix} 1 \\ 0 \end{pmatrix} = \begin{pmatrix} 3 \\ -1 \end{pmatrix}, \, \,  L \begin{pmatrix} 0 \\ 1 \end{pmatrix} = \begin{pmatrix} -4 \\ 2 \end{pmatrix}, \, \,  L \begin{pmatrix} 1 \\ 1 \end{pmatrix} = \begin{pmatrix} -1 \\ 1 \end{pmatrix}.
\]

%\[ L \left(\begin{array}{c} 1 \\ 0 \end{array} \right) = \left(\begin{array}{c} 3 \\ -1 \end{array} \right), \, \,
% L \left(\begin{array}{c} 0 \\ 1 \end{array} \right) = \left(\begin{array}{c} -4 \\ 2 \end{array} \right), \, \,
% L \left(\begin{array}{c} 1 \\ 1 \end{array} \right) = \left(\begin{array}{c} -1 \\ 1 \end{array} \right).\] 
Hence, the $\chi$-superstable configurations are: $\begin{pmatrix} 2 \\ 0 \end{pmatrix}$, $ \begin{pmatrix} 1 \\ 0 \end{pmatrix}$, and $\begin{pmatrix} 0 \\ 0 \end{pmatrix}$.
%$\left(\begin{array}{c} 2 \\ 0 \end{array} \right), \left(\begin{array}{c} 1 \\ 0 \end{array} \right)$ and $\left(\begin{array}{c} 0\\ 0 \end{array} \right)$.
Of these three configurations, $ \begin{pmatrix} 2 \\ 0 \end{pmatrix}$ is not $z$-superstable since it is in the image of $L$, $L \begin{pmatrix} 2 \\ 1 \end{pmatrix} = \begin{pmatrix} 2 \\ 0 \end{pmatrix}$.  Note this shows that $\begin{pmatrix} 2 \\ 0 \end{pmatrix}$ and $\begin{pmatrix} 0 \\ 0 \end{pmatrix}$ are equivalent under $L$ and so the $\chi$-superstable configurations are not unique per equivalence class.

\end{example}

A $z$-superstable configuration is one in which no subset of states
can fire \emph{with multiplicity} and result in a non-negative
configuration.  Again we note that $z$-superstable configurations are the same as the well-known
$\chi$-superstable configurations for undirected graphs and Eulerian directed
graphs \cite{lionel, PS}.  At the level of M-matrices (which include graph Laplacians from non-Eulerian directed graphs with global sink), $z$-superstability is the the natural notion to consider.

\subsection{z-superstables and energy minimizers}\label{z-energy}

The next results show that for an arbitrary M-matrix, $z$-superstable configurations coincide with energy minimizers (compare to Theorem 4.14 of \cite{BS}).

\begin{theorem}\label{thm1}
Let $L$ be an M-matrix. A vector $f \in \bZ^n$ with $f \ge0$ is $z$-superstable if and only if it is the minimizer of
\begin{equation*}
\min_{g\sim f, g \ge 0}  E(g).
\end{equation*}
\end{theorem}
\begin{proof}
First,  suppose that $f$ is $z$-superstable and let $g \sim  f$ with $g \ge 0$. Then we know that there exists $z \in \bZ^n$ such that $g=f-Lz$. By Lemma \ref{lemma1} $h=f-Lz^+ \ge 0$, but since $f$ is $z$-superstable then it must be that
$z^+=0$, or in other words $z \le 0$. Since by Lemma \ref{lemma2}
\begin{equation*}
E(g)=E(f)+z^tz-2z^t L^{-1} f
\end{equation*}
we have that $E(g) \ge E(f)$.

On the other hand suppose that $f$ is the minimizer. Assume for the moment $f$ is not $z$-superstable. Then this implies there exists $z \in \bZ^n$ with $z \ge 0$ and $z$ not identically zero such that $g=f-Lz \ge 0$. Since
\begin{equation*}
E(g)=E(f)-z^tz-2z^t L^{-1}g,
\end{equation*}
this implies that 
\begin{equation*}
E(g)\le E(f)-z^tz <E(f).
\end{equation*}
However, this contradicts that $f$ is the minimizer. Hence, it must be that $f$ is $z$-superstable.
\end{proof}

 We have now shown that energy-minimizers are unique up to equivalence class (see Theorem \ref{uniquem}) and that they coincide with $z$-superstable configurations.  Of course, this implies that the $z$-superstable configurations are unique up to equivalence class.

\begin{corollary}\label{uniquez}
Let $L$ be an M-matrix.  For every equivalence class defined by
$L$, there exists a unique $z$-superstable configuration. 
\end{corollary}

\noindent In the special case of non-negative row sums, another immediate corollary of the Theorems \ref{thm:same} and \ref{thm1} follows.
\begin{corollary}\label{cor1}
Let $L$ be an M-matrix with non-negative row sums. For every equivalence class, there exists a unique energy-minimizer, a unique $z$-superstable configuration, a unique $\chi$-superstable configuration all of which coincide.
\end{corollary}

Before making the connection between $z$-superstable configurations
and critical configurations, we note a few properties about
$z$-superstable configurations.  The proposition below shows that if
$f$ is a $z$-superstable configuration and any entry of $f$ is reduced
but remains non-negative, then the result is also a $z$-superstable
configuration.  In the (undirected) graphical case, $\chi$-superstable
configurations satisfy a stronger condition that for all
\emph{maximal} superstable configurations, the sum of the coordinates
is the same. The example below shows that this does not extend to
$z$-superstable configurations of arbitrary M-matrices.

\begin{proposition}
Suppose $f$ is a $z$-superstable configuration with respect to $L$ and $g \leq f$, i.e. $g$ is coordinate-wise less than or equal to $f$. Then $g$ is $z$-superstable.  Namely, $z$-superstable configurations are component-wise downward closed.  
\end{proposition}

\begin{proof}
Suppose $f$ and $g$ are as above, and $g$ is not $z$-superstable.  Then there exist $z$ such that $g - Lz \geq 0$.  Since $(Lz)_i \leq g_i \leq f_i$ for all $i$ we see that $f - Lz \geq 0$, contradicting the fact that $f$ is $z$-superstable.   
\end{proof}

\noindent On the other hand, $z$-superstable configurations do not form a \emph{pure}
order ideal of $\mathbb{N}^n$. Consider the following M-matrix,
\[
L = \begin{pmatrix}
5 & -2 \\
-4 & 3 
\end{pmatrix}.
\]
It is easily checked that $L$ has seven $z$-superstable configurations.  The two maximal configurations  under the component-wise partial order are
$\begin{pmatrix} 2 \\ 1 \end{pmatrix}$ and $\begin{pmatrix} 0 \\ 2 \end{pmatrix}$.
Hence the two maximal configurations do not have equal sums, a situation that can not occur in the graphical case, see e.g.~\cite{Merino, Tetali}.

\subsection{$\boldsymbol{z}$-superstable and critical configurations}\label{z-critical}
In this section we prove the duality pairing between $z$-superstable configurations and critical configurations.

Given a matrix $L \in \bZ^{n \times n}$ with positive diagonal entries, define $D^L \in \bZ^n$  by $D^L_i = L_{ii}-1$ for all $i$.  Namely, $D^L$ is the vector formed by taking the diagonal entries of $L$ and subtracting $1$ from each.
Recall that for a   
given  matrix $L$, a configuration $f \in \bZ^n$  is said to be stable if $f \leq D^L$. A configuration $ f \in \bZ^n$ is said to be unstable if it is not stable.

\begin{definition}
A configuration $\bc \in \bZ^n$ is a critical configuration if it is stable and if there exists a configuration $g \in \bZ^n$ with $g_i \ge L_{ii}$ for all $i$ with 
\begin{equation*}
c=g- \sum_{j=1}^k L e_{i_j},
\end{equation*}
and the requirement $$ \, g-\sum_{j=1}^\ell L e_{i_j} \ge L_{i_{\ell+1}, i_{\ell+1}}$$
 for all $\ell < k$.

\end{definition}

\noindent Interpreting the notation above, the definition states that
a configuration ${ c}$ is critical if there exists a configuration ${
  g}$ whose entries are at least as large as the diagonal of $L$ and
such that ${ g}$ can legally fire a single vertex at a time and result
in the configuration ${ c}$.  Critical configurations are also 
referred to as recurrent configurations. In the case of chip-firing on
graphs, critical configurations are often defined using the idea of
firing the sink vertex.   The definition given here is more appropriate for our
setting of M-matrices where the model does not have a site designated
as the sink.

Gabrielov established the existence of critical configurations for any avalanche-finite system.

\begin{proposition}[\cite{gabe}]
For any M-matrix, critical configurations exist and are unique per equivalence class.
\end{proposition}

 In order to prove the
connection between $z$-superstable configurations and critical
configurations for M-matrices we need two simple lemmas.

\begin{lemma}\label{big}
Suppose that $L \in \bZ^{n \times n}$ is an M-matrix. Given any vector $y \in \mathbb{Z}^{n}$ with  $y \geq 0$ there exists a vector $z \in \bZ^n$ with $z \ge 0$ such that $Lz \ge y$.
\end{lemma}
\begin{proof}
Let $g \in \mathbb{Q}^n$ be given by $g=L^{-1} y$. Since $L$ is an M-matrix and $y$ is non-negative, $g \ge 0$. Let $g_i=\frac{a_i}{b_i}$ where $a_i, b_i \in \mathbb{N}$.  Let $\lambda =b_1 b_2 \cdots b_n$ then $z=\lambda g \in \bZ^n$ with $z \ge 0$ and $Lz =\lambda y \ge y$.   
\end{proof}

\begin{lemma}\label{iter}
Suppose that $L$ is a $Z$-matrix. If $c=g-\sum_{j=1}^k L e_{i_j}$ with $c$ stable and $g$ not stable then for every $\ell$ with $g_{\ell} \ge L_{\ell \ell}$  there exists $1\le j \le k$ so that $i_j=\ell$.  
\end{lemma}
\begin{proof}
Suppose that $g_{\ell} \ge L_{\ell \ell}$ and suppose that $i_j \neq \ell$ for every $j=1, \ldots, k$. Then since $L$ is a $Z$ matrix we have
\begin{equation*}
\sum_{j=1}^{k} (L e_{i_j})_\ell \le 0,
\end{equation*}
and so
\begin{equation*}
c_{\ell} \ge g_{\ell} \ge L_{\ell \ell},
\end{equation*}
contradicting the fact that $c$ is stable.
\end{proof}

\begin{theorem}
\label{thm:duality}
Let $L$ be an M-matrix. If $f \in \bZ^n$ is $z$-superstable then $D^L-f$ is a critical configuration.  
\end{theorem}
\begin{proof}
Let $f$ be $z$-superstable. It is  not difficult to show that $D^L-f$ is stable. By  Lemma \ref{big} there exists a vector $z \ge 0$ such that $(D^L-f+Lz)_i \ge L_{ii}$ for all $i$. Set $g= D^L-f+Lz$ so that $g_{i} \ge L_{ii}$ for all $i$. Note that since $z \ge 0$ we can write $z=\sum_{j=1}^k e_{i_j}$. 
We know that 
\begin{equation*}
D^L-f=g-\sum_{j=1}^k L e_{i_j}.
\end{equation*}
The proof will be complete if we can show there exists a permutation $\sigma$ of $\{1, \ldots, k\}$ so that
\begin{equation*}
g^{\ell} := g-\sum_{j=1}^\ell L e_{i_{\sigma(j)}},
\end{equation*}
is such that 
\begin{equation}\label{eq1}
g^{\ell}_{i_{\sigma(\ell+1)}} \ge L_{i_{\sigma(\ell+1)} i_{\sigma(\ell+1)}},
\end{equation}
for $\ell=1, \ldots, k-1$.

We proceed to define the permutation $\sigma$ inductively. Suppose that we have chosen $\sigma(1), \sigma(2), \dots, \sigma(r-1)$ with $r \le k$ so that $\eqref{eq1}$ holds for $1 \le \ell \le r-2$.  We know that
\begin{equation*}
g^{r-1}= D^L-f+L \tilde{z}
\end{equation*}
or equivalently
\begin{equation*}
f-L\tilde{z} = D^L-g^{r-1},
\end{equation*}
where $$\tilde{z}=\sum_{j=1}^k Le_{i_j}-\sum_{j=1}^{r-1} L e_{i_{\sigma(j)}} \ge 0$$ and $\tilde{z} \neq 0$. Since $f$ is $z$-superstable we know that there exists a $q$ such that $(D^L-g^{r-1})_q <0$ or equivalently $g^{r-1}_q \ge L_{q q}$. Also, since $c=g^{r-1}-(\sum_{j=1}^k Le_{i_j}-\sum_{j=1}^{r-1} L e_{i_{\sigma(j)}})$,  by Lemma \ref{iter}  there exists $1 \le \sigma(r) \le k$ such that $\sigma(r) \neq \sigma(j)$ for $j=1,2, \ldots, r-1$ such that $i_{\sigma(r)}=q$. This completes the proof.
\end{proof}

Gabrielov \cite{gabe} showed that critical configurations are unique up to equivalence class for M-matrices. We have shown that $z$-superstable configurations are unique up to equivalence class  and their duals are critical configurations,  this is enough to show the following converse of Theorem \ref{thm:duality}. 

\begin{theorem}\label{converse}
Let $L \in \bZ^{n\times n}$ be an M-matrix. If $c$ is a critical configuration then $D^L-c$ is $z$-superstable. 
\end{theorem}
 
For completeness we will give an alternative proof of  Theorem \ref{converse}. To do so, we need the following known lemma which appears for example in \cite{gabe}. We also give a proof of this lemma for completeness. 
% **Say what this lemma is doing in words***
\begin{lemma}\label{fundamental}
Let $L\in \bZ^{n\times n}$ be a Z-matrix with positive diagonal entries.  Let $c$ be a stable configuration and $c=g-\sum_{j=1}^k Le_{i_j}$ with $g^\ell_{i_{\ell+1}} \ge L_{i_{\ell+1} i_{\ell+1}}$ for $\ell=1, \ldots, k-1$ where $g^\ell=g-\sum_{j=1}^\ell L e_{i_j}$. If $g-Lw$ is stable where $w \ge 0$, then  $w \ge z$ where $z= \sum_{j=1}^k e_{i_j}$.
\end{lemma}
\begin{proof}
Suppose that $w$  is not greater than $z$. Then, this implies there exists $1\le \ell \le k$ such that $w=\sum_{j=1}^{\ell-1} e_{i_j}+\tilde{w}$ with $\tilde{w} \ge 0$ and $\tilde{w}_{i_\ell}=0$. However, by our hypothesis $g^{\ell-1}_{i_{\ell}} \ge L_{i_{\ell} i_{\ell}}$. Moreover, $g-Lw=g^{\ell-1}-L \tilde{w}$, so by Lemma \ref{iter} it must be $\tilde{w}_{i_{\ell}} >0$. Hence, we reached a contradiction.  
\end{proof}

Now we turn to the proof of Theorem \ref{converse}.
\begin{proof}[Proof of Theorem \ref{converse}]
Let $c$ be a critical configuration and let $f=D^L-c$. Suppose that $f$ is not $z$-superstable.  Then, this implies there exists a vector $z \ge 0$ with $z \neq 0$ such that $f-Lz \ge 0$. Since $c$ is critical, there exists a vector $g >D^L$ such that 
\begin{equation*}
c=g- \sum_{j=1}^k L e_{i_j},
\end{equation*}
with  $g^\ell_{i_{\ell+1}} \ge L_{i_{\ell+1} i_{\ell+1}}$ for $\ell=1, \ldots, k-1$, where $g^\ell=g-\sum_{j=1}^\ell L e_{i_j}$. Setting $w=\sum_{j=1}^k e_{i_j}$,  we see that
\begin{equation*}
f-Lz=D^L-g +L(w-z).
\end{equation*}
Which gives  
\begin{equation*}
L(w-z) \ge g-D^L > 0.
\end{equation*}
Since $L$ is an M-matrix this implies that  $w-z \ge 0$. Also, we have 
\begin{equation*}
D^L \ge g-L(w-z), 
\end{equation*}
which implies $g-L(w-z)$ is stable. However, by Lemma \ref{fundamental} $w-z \ge w$. Which implies that $z \le 0$. Hence, we have reached a contradiction. 
\end{proof}

\subsection{Graph Laplacians and $G$-parking functions}

As mentioned before, the fact that $\chi$-superstable configurations
are unique (up to equivalence class) and are in simple duality with
critical configurations is well known for graph Laplacians of
undirected graphs.  
%It also follows easily from our results since the reduced graph Laplacian of any undirected graph is easily seen to be an M-matrix.

In the directed graph case, this result was first extended to Eulerian
graphs; see for example \cite{lionel, PS}.  The Eulerian condition
ensures the reduced Laplacian has non-negative column sums.  In this
case, $\chi$-superstable configurations continue to coincide with
$z$-superstable configurations.

The duality was further extended to all directed graphs with a global
sink.  The result first appears in \cite{Perlman} and later in
\cite{Perkinson} and \cite{Asadi}.  In this case, $\chi$-superstable
configurations are not the same as $z$-superstable configurations.
Critical configurations are in duality with the $z$-superstable
configurations, which are called superstable configurations in
\cite{Perkinson, Perlman} and reduced divisors in \cite{Asadi}.
If $\Delta$ is the reduced Laplacian resulting from any directed graph
with a global sink then $\Delta$ is an M-matrix, this was shown
explicitly for example in \cite{PS}.  Hence we also recover this result as
a special case of our duality pairing for any system defined by an
M-matrix.

%results show that the duality pairing extends to all directed graphs with global sink.
%%   In the directed graph case, previous results were limited to Eulerian directed graphs; see for example \cite{lionel, PS}.  The Eulerian condition ensures the reduced Laplacian has non-negative column sums.  But, this is not necessary to be an M-matrix.  If 
%% $\Delta$ is the reduced Laplacian resulting from any directed graph
%% with a global sink then $\Delta$ is an M-matrix, this was shown explicitly
%% for example in \cite{PS}.
%% Hence a special case of our results show that the duality pairing extends to all directed graphs with global sink.  The important observation is that the duality is between critical configurations and $z$-superstable configurations,  not $\chi$-superstable configurations.  
%% We end with an example of a non-Eulerian directed graph with a global
%% sink.  

As a final remark, in an attempt to clarify the literature, we relate
these notions to $G$-parking functions.  For a directed graph $G$, a
parking function is a non-negative integer sequence $(a_1, a_2, \ldots, a_n)$ such that for every subset $I \subseteq [n]$ there exists $i \in I$ such that 
$$ a_i < d_I(i),$$
where $d_I(i)$ is the number of edges from $i$ to vertices not in $I$.  
In the undirected (and directed Eulerian) graph case, $G$-parking functions,
$\chi$-superstable configurations, $z$-superstable configurations, and 
reduced divisors all coincide.

However, in the non-Eulerian directed graph case, 

\centerline{$\chi$-superstables $\neq$ $z$-superstables $=$ reduced divisors $\neq$ $G$-parking functions.}  This distinction is implicit in \cite{PS}.  We end with an explicit example illustrating the difference.

%% $\chi$-superstables do not form a system of representatives for
%% equivalence up to the image of the Laplacian.  $z$-superstables (reduced
%% divisors) and parking functions both form systems of representatives
%% but they are distinct.

%% We illustrate with the following examples.  

\begin{example}Consider the graph on $3$ vertices with directed graph
Laplacian equal to:
$$ \bordermatrix{ & 1 & 2 & \textrm{s} \cr 1 & 3 & -3 & 0 \cr 2 & -1
  & 2 & -1 \cr \textrm{s} & 0 & 0 & 0}. $$ Vertex $1$ has three edges
directed to vertex $2$.  Vertex $2$ has a single edge to vertex $1$
and a single edge to the sink.  The sink has no outgoing edges. 
The transpose of the reduced graph Laplacian is:
$$ L = \begin{pmatrix}
3 & -1\\
-3 & 2
\end{pmatrix}.
$$ It is not hard to check that all four $0/1$-vectors of length two
are $\chi$-superstable for this graph.  On the other hand, the all
ones vector is not $z$-superstable as it is equal to $L \cdot
(1,2)^T$.  In particular, the all ones configuration is equivalent to
the all zeros configuration.  For this graph, $D^L = (2,1)$ and hence
the critical configurations are $(2,1)$, $(1,1)$, and $(2,0)$, the $z$-superstables are $(0,0)$, $(1,0)$, $(0,1)$.  It is
easily checked that the $G$-parking functions are $(0,0)$, $(1,0)$,
$(2,0)$.

\end{example}

{\bf Acknowledgements} The authors thank David Perkinson for many helpful discussions and an anonymous reviewer for a thoughtful question concerning energy forms.  

\section{Appendix}

Here we show that z-superstable configurations are also the  minimizers of a more general class of energies. We consider any energy of the form
\begin{equation}\label{generalenergyappnedix}
E(q)=\sum_{i=1}^n \phi_i((L^{-1} q)_i),
\end{equation}
where the functions $\phi_i: \mathbb{R} \rightarrow \mathbb{R}$ are non-negative and strictly increasing.

We first prove that there is a unique $E(q)$-minimizer per equivalence class. The minimization problem remains the same.

Given  $f \in \mathbb{Z}^n$ with  $f \ge 0$ consider the following problem:
\begin{equation}\label{minp}
\min_{g\sim f, g \ge 0}  E(g).
\end{equation}

The following result is the generalization of Theorem \ref{uniquem}. In fact, the wording of the statement is exactly the same except that we now consider the energy \eqref{generalenergyappnedix}. Also, the reader will notice that the beginning of the proof is identical to that of Theorem \ref{uniquem}.
% with some of the sentences being identical. 

\begin{theorem}
Let $L$ be  an M-matrix. For every configuration $f$, there exists a unique energy minimizer equivalent to $f$. Namely, for every configuration $f$, there exists a unique solution to problem \eqref{minp}. 
\end{theorem} 

\begin{proof}
Suppose that $g \sim f$ and $w\sim f$ with $ g,w \ge 0$ are both minimizers to problem \eqref{min}. We will show that $g=w$.
Because $g$ is equivalent to $w$,  there exists $z$ such that $g=w-Lz$ for some $z \in \bZ^n$. By Lemma \ref{lemma1} we know that $h=w-L z^+ \ge 0$ and of course $h \sim w \sim f$. We have that $H=W-z^+$ where $H=L^{-1}h$ and $W=L^{-1} w$. Since $L$ is an M-matrix the entries of $L^{-1}$ are non-negative and so $H, W >0$.  Note that $0 \le H_i \le W_i$ and since $\phi_i$ is strictly increasing we have $\phi_i(|H_i|) \le  \phi_i( |W_i| )$ for all $1 \le i \le n$ with strict inequality whenever $z_i^+>0$. 

Therefore,  
\begin{equation*}
E(h)=\sum_{i=1}^n \phi_i(|H_i|) \le \sum_{i=1}^n \phi_i(|W_i|)=E(w),
\end{equation*}
with strict inequality if at least one $z_i^+$ is greater than zero.  Since $w$ is a minimizer it must be that $z^+$ is identically zero. In other words, $z \le 0$. If we let $G=L^{-1}g$ then we have $G=W-z$ and therefore we have that $G_i \ge W_i \ge 0$ with strict inequality whenever $z_i<0$. Hence, 

\begin{equation*}
E(g)=\sum_{i=1}^n \phi_i(|G_i|) \ge \sum_{i=1}^n \phi_i(|W_i|)=E(w),
\end{equation*}
with strict inequality if  at least one $z_i$ is less than zero. Since $g$ is a minimizer it must be that $z$ is identically zero. That is, $g=w$.  
\end{proof}

The following result is the generalization to Theorem \ref{thm1}.

\begin{theorem}
Let $L$ be an M-matrix. A vector $f \in \bZ^n$ with $f \ge0$ is $z$-superstable if and only if it is the minimizer of
\begin{equation*}
\min_{g\sim f, g \ge 0}  E(g).
\end{equation*}
\end{theorem}
\begin{proof}
First,  suppose that $f$ is $z$-superstable and let $g \sim  f$ with $g \ge 0$. Then we know that there exists $z \in \bZ^n$ such that $g=f-Lz$. By Lemma \ref{lemma1} $h=f-Lz^+ \ge 0$, but since $f$ is $z$-superstable then it must be that $z^+=0$, or in other words $z \le 0$. We have $G=F-z$ where $G=L^{-1} g$ and $F= L^{-1} f$ and so  $G_i \ge F_i \ge 0$. Therefore, we see that $\phi_i(|G_i|) \ge \phi_i(|F_i|)$ for every $1 \le i \le n$. Hence, we have $E(g) \ge E(f)$, and so $f$ is a minimizer. 

On the other hand, suppose that $f$ is the minimizer. Assume for the moment $f$ is not $z$-superstable. Then this implies there exists $z \in \bZ^n$ with $z \ge 0$ and $z$ not identically zero such that $g=f-Lz \ge 0$. Since $G=F-z$ it must be that $0 \le G_i \le F_i$ for every $i$ and $G_i < F_i$ for at least one $i$. Hence, $\phi_i(|G_i|) \le \phi_i(|F_i|)$ for every $i$ and $\phi_i(|G_i|) < \phi_i(|F_i|)$ for at least one $i$. Therefore, $E(g)< E(f)$, but this contradicts that $f$ is a minimizer. Hence, it  must be that $f$ is $z$-superstable.
\end{proof}

%% ------------------------------------------------------------------

\end{document}